\documentclass[a4paper, 11pt, underruledhead]{paper}
\usepackage{enumerate, math}
\usepackage[mathscr]{eucal}

\newif\ifrach
\rachtrue
\rachfalse

\title[Affine polar spaces derived from symplectic spaces]{Affine polar spaces derived %
from symplectic spaces, their geometry and representations: alternating semiforms}
\author{K. Pra{\.z}mowski, M. {\.Z}ynel}

\def\afsempol{affine semipolar space}

\DeclareMathOperator{\Dom}{Dom}

\DeclareMathOperator{\Tr}{Tr}

\begin{document}


\def\inc{\mathrel{\rule{2pt}{0pt}\rule[-0.8pt]{1pt}{2ex}\rule{2pt}{0pt}}}
\def\ninc{\mathrel{\not\mkern3mu\inc}}

\def\bisect(#1,#2){{\bf L}_{\sf t}\binom{#1}{#2}}
\def\bisecm(#1,#2){{\bf L}_{\sf m}\binom{#1}{#2}}

\def\GL{\mathop{\mathit{GL}}}
\def\AF{\mathop{\mathrm{AF}}}
\def\adjac{\mathrel{\sim}}

\def\toadjac{\mathrel{\lower.35ex\hbox{\baselineskip-3pt\lineskip-3pt\vbox{\hbox{$\sim$}\hbox{$\sim$}}}}}
\let\doadjac\toadjac
\def\soadjac{\lower.25ex\hbox{\scriptsize\baselineskip-2.2pt\lineskip-2.2pt\vbox{\hbox{$\sim$}\hbox{$\sim$}}}}
\def\ssoadjac{\lower.21ex\hbox{\tiny\baselineskip-2.2pt\lineskip-2.2pt\vbox{\hbox{$\sim$}\hbox{$\sim$}}}}
\def\oadjac{\mathchoice{\doadjac}{\toadjac}{\soadjac}{\ssoadjac}}

\def\penc{{\bf p}}
\def\pencx{{\bf p}^\ast}

\def\toadjac{\mathrel{\lower.35ex\hbox{\baselineskip-3pt\lineskip-3pt\vbox{\hbox{$\sim$}\hbox{$\sim$}}}}}
\let\doadjac\toadjac
\def\soadjac{\lower.25ex\hbox{\scriptsize\baselineskip-2.2pt\lineskip-2.2pt\vbox{\hbox{$\sim$}\hbox{$\sim$}}}}
\def\ssoadjac{\lower.21ex\hbox{\tiny\baselineskip-2.2pt\lineskip-2.2pt\vbox{\hbox{$\sim$}\hbox{$\sim$}}}}
\def\badjac{\mathchoice{\doadjac}{\toadjac}{\soadjac}{\ssoadjac}}
\newcommand{\minus}{{\bf -}}

\def\txminus{\lower-.35ex\hbox{\baselineskip4pt\lineskip-7pt\vbox{\hbox{$\minus$}\hbox{$\sim$}}}}
\let\dxminus\txminus
\def\sxminus{\lower-.25ex\hbox{\scriptsize\baselineskip-2pt\lineskip-7pt\vbox{\hbox{$\minus$}\hbox{$\sim$}}}}
\def\ssxminus{\lower.21ex\hbox{\tiny\baselineskip-1.4pt\lineskip-5pt\vbox{\hbox{$\minus$}\hbox{$\sim$}}}}
\def\baminus{\mathchoice{\dxminus}{\txminus}{\sxminus}{\ssxminus}}

\def\colin{{\mbox{\boldmath$L$}}}
\def\colind{{\mbox{\boldmath$L$}^{\propandimprop}}}

\def\lines{{\cal L}}
\def\izolines{{\cal G}}
\def\peki{{\cal P}}
\def\pekix{{\cal P}^\ast}
\def\pekid{{\cal P}^{\propandimprop}}
\def\subm{\widetilde{\sub}}
\def\upadjac{\mathrel{\adjac^{\mbox{\tiny\boldmath$+$}}}}
\def\downadjac{\mathrel{\adjac_{\mbox{\tiny\boldmath$-$}}}}
\def\paradjac{\mathrel{\adjac^{\shortparallel}}}
\def\afupadjac{\mathrel{\baminus^{\mbox{\tiny\boldmath$+$}}}}
\def\afbotadjac{\mathrel{\baminus_{\mbox{\tiny\boldmath$-$}}}}
\def\fixproj{\mbox{\boldmath$\goth P$}}
\def\fixaf{\mbox{\boldmath$\goth A$}}
\def\fixpolar{\mbox{\boldmath$\goth Q$}}
\let\fixpol\fixpolar
\def\fixafpolar{{\mbox{\boldmath$\goth U$}}}
\let\fixafpol\fixafpolar
\def\srodek{\mbox{\boldmath$U$}}
\def\pls{partial linear space}

\def\vgen#1{\gen{[#1]}}
\def\bvgen#1{\bgen{\bigl[#1\bigr]}}
\def\agen#1#2{{\mbox{{\boldmath$[$}}{#1}\mbox{{\boldmath$]$}}_{{#2}}}}

\def\starof{{\mathrm{S}}}
\def\polarstar{{\starof_0}}
\def\topof{{\mathrm{T}}}
\def\stars{{\cal S}}
\def\starsx{{\cal S}^\ast}
\def\starsm{{\cal S}^\vee}
\def\tops{{\cal T}}
\def\topsx{{\cal T}^\ast}

\def\ProjectiveSpSymb{\mathbf{P}}
\def\PencSpace(#1,#2){\ensuremath{\ProjectiveSpSymb_{#1}(#2)}}
\def\PencSpacex(#1,#2){\ensuremath{\ProjectiveSpSymb^\ast_{#1}(#2)}}
\def\PencSpaced(#1,#2){\ensuremath{\ProjectiveSpSymb^{\propandimprop}_{#1}(#2)}}
\def\GrassmannSpSymb{\mathbf{G}}
\def\GrasSpace(#1,#2){\ensuremath{\GrassmannSpSymb_{#1}(#2)}}
\def\AffineSpSymb{\mathbf{A}}
\def\AfSpace(#1,#2){\ensuremath{\AffineSpSymb_{#1}(#2)}}
\def\KwadrSpSymb{\mathbf{Q}}
\def\KwadrSpace(#1,#2){\ensuremath{\KwadrSpSymb_{#1}(#2)}}
\def\AfPolSpSymb{\mathbf{U}}
\def\AfpolSpace(#1,#2){\ensuremath{\AfPolSpSymb_{#1}(#2)}}
\def\AfpolSpacex(#1,#2){\ensuremath{\AfPolSpSymb^\dagger_{#1}(#2)}}

\def\IsoSpace{{\mbox{\boldmath$\goth Q$}}}
\def\AfPolar{{\mbox{\boldmath$\goth U$}}}
\let\Afpolar\AfPolar
\def\AfPolarx{{\mbox{\boldmath$\goth U$}}^{\propandimprop}}

\def\Quadr{{\mathrm{Q}}}
\def\LineOn(#1,#2){\overline{{#1},{#2}\rule{0em}{1,5ex}}}
\def\btheta{\mbox{\boldmath$0$}}
\def\semform{\varrho}
\def\sematlas{\phi}
\def\hipa{{\mathscr H}}

\def\vecprod{\mbox{\boldmath$\,\times\,$}}
\def\dimcel{\nu}
\def\Dom{\mathrm{Dom}}

\def\axC{{\upshape\sffamily C}}
\def\axA{{\upshape\sffamily A}}
\def\axD{{\upshape $(\ast\ast)$}}
\def\axE{{\upshape $(\ast)$}}
\def\refaxC#1{{\upshape\sffamily C\ref{#1}}}
\def\refaxA#1{{\upshape\sffamily A\ref{#1}}}

\newenvironment{ctext}{%
  \par
  \smallskip
  \centering
}{%
 \par
 \smallskip
 \csname @endpetrue\endcsname
}

\def\rref#1{\ref{#1})}

\newcounter{exmc}
\def\theexmc{\Alph{exmc}}
\def\labelexmc{\theexmc}
\newenvironment{exmc}[1]{%
\refstepcounter{exmc}
\trivlist\item
{\bfseries\scshape Example-continuation \ref{#1}-\labelexmc}
}{%
\endtrivlist}

\def\exmcref#1#2{\ref{#1}-\ref{#2}}

\newcounter{kontki}
\setcounter{kontki}{0}

\newenvironment{extexm}[1]{\refstepcounter{kontki}
\par\noindent
{\sc\bfseries Example \ref{#1}.\thekontki, a continuation of \ref{#1}}
}{%
\par}

\def\extref#1#2{\ref{#1}.\ref{#2}}



\maketitle

\begin{abstract}
  Deleting a hyperplane from a polar space associated with a symplectic polarity
  we get a specific, symplectic, affine polar space.  Similar geometry, called
  an \afsempol\ arises as a result of generalization of the notion of an
  alternating form to a semiform. Some properties of these two geometries are
  given and  their automorphism groups are characterized.

  \begin{flushleft}
  \end{flushleft}
\end{abstract}


\section*{Introduction}

In \cite{cohenshult} affine polar spaces are derived from polar spaces the same
way as affine spaces are derived from projective spaces, i.e. by deleting a hyperplane
from a polar space embedded into a projective space.
So, affine polar spaces  
(aps'es, in short)
are embeddable in affine spaces and this let us think of them
as of suitable   
reducts of affine spaces.

In general we have two types of affine polar spaces.  Structures of the first
type are  associated with polar spaces determined by sesquilinear forms; one can
loosely say: these are ``stereographical projections of quadrics".  They can
also be thought of as determined by sesquilinear forms defined on vector spaces
which represent respective affine spaces. These structures and adjacency of
their subspaces were studied in \cite{afpolar}. Contrary to \cite{cohenshult},
in this approach Minkowskian geometry is not excluded. In particular, the result
of \cite{afpolar} generalizes Alexandrov-Zeeman theorems originally concerning
adjacency \emph{of points} of an affine polar space (cf. \cite{Alexandrov},
\cite{Zeeman}). 

\medskip
The second class of affine polar spaces, which is included in \cite{cohenshult}
but is excluded from \cite{afpolar}, consists of structures associated with
polar spaces determined by symplectic polarities. The aim of this paper is to
present  in some detail the geometry of the structures  in this class  from view of
the affine  space in which they are embedded.
The position of the class of thus obtained structures -- let us call them symplectic
affine polar spaces -- is in many points a particular one.
\par
Firstly, symplectic affine polar spaces are associated
with null-systems, quite well known polarities in
projective spaces with all points selfconjugate. 
So, symplectic aps'es have famous parents.
Moreover, in each even-dimensional 
pappian projective space such a (projectively unique) polarity exists.
Thus a symplectic aps is not an exceptional space,
but conversely, it is also a ``canonical"
one in each admissible dimension. 
\par
A second argument refers to the position of the class of symplectic aps'es in the 
class of all aps'es. As an affine polar space is obtained by deleting a hyperplane
from a polar space, while the latter is realized as a quadric in a metric projective
space, the derived aps appears as a fragment of the derived affine space.
If the underlying form that determines the polar space is symmetric then the 
corresponding aps can be realized {\em on} an affine space in one case only
-- when the deleted hyperplane is a tangent one.
And then the affine space in question is constructed {\em not} as a reduct
of the surrounding projective space but as a derived space as it is 
done in the context of chain geometry (cf. \cite{herz}, \cite{benz}).
Moreover, such an aps can be represented without the whole machinery of polar
spaces: it is the structure of isotropic lines of a metric affine space.
\newline
The only case when the point set of the reduct of a polar space {\em is} the 
point set of an affine space arises when we start from a null system 
i.e. in the case considered in the paper. But such an aps is not associated with
a metric affine space i.e with a vector space endowed with a nondegenerate
bilinear symmetric form. What is a natural analytic way in which a symplectic
aps can be represented, when its point set is represented via a vector space?
A way to do so is proposed in our paper: to this aim we consider a ``metric'',
a binary scalar-valued operation defined on vectors. It is not a metric,
in particular, it is not symmetric, and it is not invariant under affine
translations. Nevertheless, it suffices to characterize respective geometry.
\par
Symplectic aps'es have famous parents but they have also remarkable relatives.
Although the ``metric'': the analytical characteristic invariant of symplectic
aps'es is not a form, it is closely related to forms.
Loosely speaking, it is a sum of an alternating form $\eta$ defined on a subspace
and an affine vector atlas defined on a vector complement of the domain of $\eta$. 
Immediate generalization with `an alternating map'
substituted in place of `an alternating form' comes to mind.
Such a definition of a map may seem artificial, the resulting maps, which we
call {\em semiforms}, have quite nice synthetic characterization though.
Their basic properties are established in Section~\ref{sec:semiforms}.
A symplectic `metric' 
appears to be merely a special instance of such a general
definition and many problems concerning 
it (so as to mention a characterization of the automorphism group)
can be 
solved in this general setting easier.
To illustrate and to motivate such a general definition we show in \ref{exm:prod} 
a semiform associated with a vector product, that yields also an interesting
geometry. On the other hand, this geometry has close connections (see \exmcref{exm:prod}{cont2})
with a class of hyperbolic polar spaces.
\newline
A semiform induces an incidence geometry that we call an {\em \afsempol} 
(cf. \eqref{eq:badjac} and \eqref{eq:izolines}).
It is a $\Gamma$-space
with affine spaces as its singular subspaces (cf. \ref{thm:Gamma}), and with generalized null-systems
comprised by lines and planes through a fixed point 
(cf. \ref{lem:geowiazki:0}; 
comp. a class with similar properties considered in \cite{cyup-pas}).
In the paper we do not go any deeper into details of neither 
geometry of semiforms nor geometries other than symplectic aps'es.
We rather concentrate on ``aps'es and  around''. 
\par
Finally, we pass to our third group of arguments: 
that geometry of symplectic aps'es 
is interesting on its own right. 
Geometry of affine polar spaces is, by definition, an incidence geometry i.e. 
an aps is (as it was defined both in \cite{cohenshult} and \cite{afpolar})
a partial linear space: a structure with points and lines.
From the results of \cite{cohen}  we get that geometry
of symplectic affine polar spaces can be also formulated in terms of 
binary collinearity of points -- an analogue of the Alexandrov-Zeeman Theorem.
A characterization of aps'es as suitable graphs is not known, though.
\newline
The affine polar spaces associated with metric affine spaces (as it was
sketched above) can be, in a natural consequence, characterized in the ``metric''
language of line orthogonality or equidistance relation inherited from the 
underlying metric affine structure. 
It is impossible to investigate a line orthogonality imposed on an affine structure
so as it gives rise to a symplectic aps.
However, in case of a symplectic aps
a ``metric'' mentioned above determines an ``equidistance'' relation
which can be used as a primitive notion to characterize the geometry.
There is no general commonly accepted axiom system of a weak equidistance relation
(of a congruence of segments, in other words).
A very natural one, that characterizes metric affine spaces is presented in 
\cite{szroeder}. Roughly speaking, 
in accordance with that approach a congruence of segments
is an equivalence relation on pairs of points such that
bisector hyperplanes are really affine hyperplanes.
But these properties are met by our ``symplectic equidistance'' as well.
The difference is that a segment and its translate need not be congruent under
our equidistance.
In this paper 
we do not intend to give a characterization of symplectic aps'es in the language
of equidistance. Nevertheless, we think it is worth to stress on that this
is also a possible language for this geometry and to indicate similarities
and dissimilarities between our equidistance and that used in metric affine geometry.
Since our equidistance is not commutative we have two types of bisectors 
(cf. \eqref{eq:bisectors}) and thus two types of symmetry under a hyperplane 
(cf. \ref{thm:symmetry}). The first type of a relation of being symmetric wrt. a hyperplane
is related to translations and the other type is a central symmetry.


\section{Definitions and preliminary results}\label{sec:defy}

Recall that the affine space 
$\AfSpace({},{\field V})$ 
defined over a vector space $\field V$ has the vectors of $\field V$ as its points
and the cosets of the $1$-dimensional subspaces of $\field V$ as its lines.

We write $\tau_\omega$ for the (affine) translation on the vector $\omega$,
$\tau_\omega(x) = x+ \omega$.


\subsection{Polar spaces}\label{subsec:def:polary}

Let $\field W$ be a vector space over a (commutative) field $\goth F$
with characteristic $\neq 2$ and let $\xi$ be a nondegenerate
bilinear reflexive form defined on $\field W$.
Assume that the form $\xi$ has finite index $m$ and $n = \dim({\field W})$.
We will write $\Sub({\field W})$ for the class of all vector subspaces of $\field W$
and 
$\Sub_k({\field W})$ for the class of all $k$-dimensional subspaces.
In the projective space 
$\fixproj = \struct{\Sub_1({\field W}),\Sub_2({\field W}),\subset}$ 
the form $\xi$ determines the polarity $\delta = \delta_\xi$.
We write $\Quadr(\xi)$ for the class of isotropic subspaces of $\field W$:
\begin{equation*}
  \Quadr(\xi) = \left\{ U\in\Sub({\field W})\colon \xi(U,U)=0 \right\};\quad
  \Quadr_k(\xi) = \left\{ U\in\Quadr(\xi)\colon \dim(U) = k \right\}.
\end{equation*}
Assume that $m \geq 2$. The structure 
\begin{ctext}
  $\KwadrSpace(\xi,{\field W}) := \struct{\Quadr_1(\xi),\Quadr_2(\xi),\subset}$
\end{ctext}
is referred to as
the {\em polar space} determined by $\delta$ in $\fixproj$.


\subsection{Hyperbolic polar spaces and their reducts}\label{sssec:hyperbolic}

This section may look superfluous from view of symplectic polar spaces but 
it is used later in an example which justifies our general construction of semiforms.

Now let $\xi$ be symmetric and $\perp = \perp_\xi$ be the orthogonality determined
by $\xi$ on $Y := W\times W$. Set ${\field Y} := {\field W}\oplus \field{W}$,
$Z := W\times\Theta$, and $H:= \{ [u,v]\in Y\colon u\perp v\}$.
Then there is a nondegenerate form $\zeta$ on $Y$ such that 
$\Sub_1(H) = \Quadr_1(\zeta)$. 

Note that $Z$ is a maximal isotropic subspace of $(Y,\perp)$ and thus
the geometry $\fixpol := \KwadrSpace(\zeta,{\field Y})$ is a, so called, 
hyperbolic polar space 
(or a hyperbolic quadric following \cite[Sec. 1.3.4, p. 30]{Pasini},
cf. also \cite{cohen}).

Now let $\cal Z$ be a maximal (i.e. a $(n-1)$-dimensional) singular subspace of 
a hyperbolic polar space $\fixpol$ of index $n-1$
and let 
${\goth R} = {\goth R}({\goth Q}, {\cal Z})$ 
be the structure obtained by deleting the subspace $\cal Z$ from 
$\fixpol$. 
In particular we write 
${\goth R}({\field W},\xi) = {\goth R}(\KwadrSpace(\zeta,{\field Y}),\Sub_1(Z))$.

\begin{thm}\label{thm:hippolreduct}
  The hyperbolic polar space $\fixpol$ is definable in its reduct\/ $\goth R$.
\end{thm}

\begin{proof}
  We need to recover from $\goth R$ the points and lines of $\cal Z$ which are missing
  to get $\fixpol$.
  Let $\cal C$ be the family of the maximal singular subspaces of $\fixpol$ and
  $\cal R$ be the family of maximal singular subspaces of $\goth R$.
  It is seen that 
  ${\cal R} = \{ {\cal X}\setminus{\cal Z}\colon {\cal X} \in {\cal C} \}$,
  and thus each element of $\cal R$ carries the geometry of a slit 
  space (cf. \cite{KM67}, \cite{KP70}).
  Write 
  ${\cal R}_1 = \{ {\cal X}\setminus {\cal Z}\colon \dim({\cal X}\cap{\cal Z}) = n-2 \}$
  and
  ${\cal R}_0 = \{ {\cal X}\setminus {\cal Z}\colon \dim({\cal X}\cap{\cal Z}) = 0 \}$.
  So,
  \begin{itemize}\def\labelitemi{--}\itemsep-2pt
  \item 
    ${\cal R}_1$ consists of the elements of $\cal R$ which carry the affine 
    geometry. Each  ${\cal J}\in{\cal R}_1$ determines on $\cal Z$ 
    a $(n-2)$-subspace of its ``improper'' points.
  \item
    ${\cal R}_0$ consists of the elements of $\cal R$ which carry the geometry 
    of a punctured projective space. 
    Each  ${\cal J}\in{\cal R}_0$ determines on $\cal Z$ 
    a point: its ``improper'' point.
  \end{itemize}
  \par
  Next, we need some theory of hyperbolic polar spaces (cf. \cite[Sec. 1.3.6, p. 35]{Pasini}).
  In the class $\cal C$ we define the relation: 
  ${\cal X}_1 \approx {\cal X}_2$ iff 
  $ 2 | (\dim({\cal X}_1) - \dim({\cal X}_1 \cap  {\cal X}_2)) = 
    n-1-\dim({\cal X}_1 \cap  {\cal X}_2)$;
  it is an equivalence relation.
  Take 
  ${\cal J}_0\in{\cal R}_0$,  ${\cal J}_0 = {\cal X}_0\setminus{\cal Z}$, and
  ${\cal J}_1\in{\cal R}_1$,  ${\cal J}_1 = {\cal X}_1\setminus{\cal Z}$
  where ${\cal X}_0, {\cal X}_1\in {\cal C}$. Set
  $J_0 = {\cal X}_0\cap{\cal Z}$,
  $J_1 = {\cal X}_1\cap{\cal Z}$. Then
  \begin{equation}\label{eq:wzorek}
    J_0 \inc J_1 \text {\quad iff \quad there exists a line } L 
    \text{ of } {\goth R} \text{ such that } L\subset {\cal J}_0\cap {\cal J}_1.
  \end{equation}
  We write ${\cal J}_0 \inc {\cal J}_1$ when the right-hand of \eqref{eq:wzorek} holds.
  Assume first that $J_0\inc J_1$. Then $J_0\inc {\cal X}_0, {\cal X}_1$. 
  From assumptions 
  ${\cal Z}\not\approx{\cal X}_1$. In case $2\div n-1$ we have 
  ${\cal Z}\approx{\cal X}_0$, so ${\cal X}_0\not\approx {\cal X}_1$.
  Hence $\dim({\cal X}_0\cap {\cal X}_1) > 0$ which means that
  ${\cal X}_0, {\cal X}_1$ share a line $L'$. 
  Clearly, $L' \not\subset {\cal Z}$, so $L = L'\setminus{\cal Z}$ is a required line
  of $\goth R$.  
  In case $2\nmid n-1$ we get that ${\cal Z}\not\approx{\cal X}_0$
  and thus ${\cal X}_0\approx {\cal X}_1$ and the above reasoning can be applied
  again.
  Conversely, assume that ${\cal J}_0, {\cal J}_1$ share a line $L$; then
  $L = L'\setminus{\cal Z}$ for a line of $\fixpol$ contained in  
  ${\cal X}_0, {\cal X}_1$. 
  Take $p = L'\cap{\cal Z}$: the suitable improper point. 
  So $p\inc J_1$. On the other hand $p$ is the unique improper
  point of ${\cal J}_0$, i.e. $p = J_0$.
  Now, consider the relation $\simeq$ defined in the class ${\cal R}_0$
  as follows
  \begin{equation}\label{eq:defpointsZ}
    {\cal J}_0' \simeq {\cal J}_0''\text{ iff } 
      \text{ for all } {\cal J}_1\in{\cal R}_1 \text{ we have }
      ({\cal J}_0' \inc {\cal J}_1 \iff {\cal J}_0'' \inc {\cal J}_1).
  \end{equation}
  Note that this is an equivalence relation and its equivalence classes can be identified 
  with points on $\cal Z$. In turn, as $\fixpol$ is of type {\sffamily D} (i.e. it is hyperbolic),
  the elements of ${\cal R}_1$ can be identified 
  with the hyperplanes of $\cal Z$ quite naturally. That way, in terms of $\goth R$, we 
  get an incidence structure with points and hyperplanes of $\cal Z$. Using standard
  methods we are able to recover lines of $\cal Z$ from this incidence structure
  which makes the proof complete. 
\end{proof}


\subsection{Symplectic affine polar spaces}\label{subsec:def:apsy}

From now on $\xi$ is a nondegenerate symplectic form of index $m$. Then 
$n = \dim({\field W}) = 2m$.
Assume that $m \geq 2$. The polar space
\begin{ctext}
  $\fixpol := \KwadrSpace(\xi,{\field W}) = \struct{\Quadr_1(\xi),\Quadr_2(\xi),\subset}$
\end{ctext}
is frequently referred to as a {\em null system}
 (cf. \cite{null-system}, \cite[Vol. 2, Ch. 9, Sec. 3]{hodge-pedoe}).
Since $\xi$ is symplectic, $\Quadr_1(\xi) = \Sub_1({\field W})$ so,
the point sets of $\fixpol$ and of $\fixproj$ coincide.

Let $\hipa_0$ be a hyperplane of $\fixpol$ (cf. \cite{cohenshult}); 
then $\hipa_0$ is determined by a hyperplane 
$\hipa$ of $\fixproj$; on the other hand $\hipa$ is a polar hyperplane of a point 
$\srodek$ of $\fixproj$ i.e. $\hipa = \srodek^\perp$. 
Finally, $\hipa_0 = \hipa$ is the set of all
the points that are collinear in $\fixpol$ with  the point $\srodek$ of $\fixpol$.
The 
{\em affine polar space $\fixafpol$ derived from ($\fixpol$,$\srodek$)} 
is the restriction 
of $\fixpol$ to the complement of $\hipa$; in view of the above the point set
of $\fixafpol$ is the point set of the affine space $\fixaf$ obtained from
$\fixproj$ by deleting its hyperplane $\hipa$.
The set $\izolines$ of lines of $\fixafpol$ is a subclass of the set $\lines$
of the lines of $\fixaf$.
Moreover, the parallelism of the lines in $\izolines$ defined as in
\cite{cohenshult} (two lines are parallel iff they intersect in $\hipa_0$)
coincides with the parallelism of $\fixaf$ restricted to $\izolines$.
Clearly, not every line of $\fixproj$ 
that is not contained in $\hipa$ and which crosses $\hipa$
in a point $U'$ is isotropic. Moreover, none of the lines of $\fixproj$ through $\srodek$ 
which is not contained in $\hipa$ is isotropic.  
For this reason, in every direction of\/ $\fixaf$, except the one determined by $\srodek$, 
there is a pair of parallel lines in $\fixaf$ such that one of them is isotropic and the
other is not. In this exceptional direction no line is isotropic.

In \cite{afpolar} affine polar spaces determined in metric affine spaces
associated with symmetric forms were studied. 
Slightly similar interpretation of $\fixafpol$ can be given here as well.

Recall that there is a basis of $\field W$ in which the form $\xi$
is given by the formula
\begin{ctext}
  $\xi(x,y) = (x_1y_2 - x_2y_1) + (x_3y_4 - x_4y_3) + \ldots
  = \sum_{i=1}^{m} (x_{2i-1}y_{2i} - x_{2i}y_{2i-1})$.
\end{ctext}
We write $\gen{u, v,\dots}$ for the vector subspace spanned by $u, v, \dots$
and $[x,y,z,...]$ for the vector with coordinates $x,y,z,...$
(in some cases $x$, $y$ ... may be vectors too).

Let us take $\srodek = \vgen{0,1,0,\ldots,0}$; then $\hipa$ is characterized by the condition
$\vgen{x_1,\ldots,x_n} \subset \hipa$ iff $x_1 = 0$.
We write
$\field V$ for the subspace of $\field W$
characterized by $x_1 = x_2 = 0$; note that the restriction $\eta$ of $\xi$ to $\field V$
is also a nondegenerate symplectic form. 
We can write ${\field W} = {\goth F}\oplus{\goth F}\oplus{\field V}$
and then for scalars $a_1,a_2,b_1,b_2$ and vectors $u_1,u_2$ of $\field V$
we have
\begin{equation}\label{eq:form2}
  \xi([a_1,b_1,u_1],[a_2,b_2,u_2]) = a_1b_2 - a_2b_1 + \eta(u_1,u_2).
\end{equation}
Moreover, $\fixaf = \AfSpace({}, {\field Y})$ where ${\field Y} = {\goth F}\oplus{\field V}$.
A point $[a,u]$ ($a$ in $\goth F$, $u \in {\field V}$)
of \fixaf\ can be identified with the subspace $\vgen{1,a,u}$ of $\field W$,
and the (affine) direction of the line $[a,u] + \vgen{b,w}$ is identified
with the (projective) point $\vgen{0,b,w}$.
\begin{lem}\label{lem:anal:izolines}
  Let $L = [a,u] + \vgen{b,w}$ with $a,b\in{\goth F}$, $u,w \in {\field V}$ 
  be a line of\/ \fixaf.  
  Then
  \begin{equation}\label{anal:izolines}
    L \in \izolines \iff \eta(u,w) = -b.
  \end{equation}
\end{lem}
\begin{proof}
  The projective completion $\overline{L}$ of $L$ in $\fixproj$ is determined by
  the pair $U_1 = \vgen{1,a,u}$, $U_2 = \vgen{0,b,w}$ of points of $\fixproj$. 
  Then $L \in\izolines$ iff $\xi(U_1,U_2) = 0$.
  With \eqref{eq:form2} we get the claim.
\end{proof}
For $[a_1,u_1],[a_2,u_2]\in{\field Y}$ we define
\begin{equation}\label{def:metryka}
  \rho([a_1,u_1],[a_2,u_2]) := \eta(u_1,u_2) - (a_1 - a_2).
\end{equation}
\begin{lem}\label{lem:anal:colin}
  Let $p_1 = [a_1,u_1]$, $p_2 = [a_2,u_2]$ with $a_1,a_2\in{\goth F}$,
  $u_1,u_2\in{\field V}$
  be a pair of points of\/ \fixaf. Then
  \begin{equation}\label{anal:colin}
     p_1, p_2 \text{ are collinear in }\fixafpol \iff 
     \eta(u_1,u_2) = a_1 - a_2
     \iff
     \rho(p_1,p_2) = 0.
  \end{equation}
\end{lem}
\begin{proof}
  As in \ref{lem:anal:izolines}, we embed given points into $\fixproj$; then 
  $p_i$ corresponds to $U_i = \vgen{1,a_i,u_i}$.
  Since $p_1,p_2$ are collinear iff the projective line
  which joins $U_1,U_2$ is in $\fixpol$ we get that $p_1,p_2$ are 
  collinear iff $\xi(U_1,U_2)= 0$. With \eqref{eq:form2}
  we have the claim.
\end{proof}

In what follows we write $a\adjac b$ if points $a,b$ of \fixaf\ 
are collinear in \fixafpol. 

From \ref{lem:anal:izolines} and \ref{lem:anal:colin}
we learn that the affine polar space $\fixafpol$ can be defined
entirely in terms of a vector space $\field V$ over $\goth F$ and 
a nondegenerate symplectic form $\eta$ on $\field V$.

\medskip

Note that the surrounding affine space \fixaf\ is definable in terms of the 
geometry of \fixafpol. 
The result is a simple consequence of elementary
properties of (symplectic) polar spaces, but it is important from the view
of ``foundations'': {\em the geometry of symplectic aps'es can be expressed in
the language of the relation $\adjac$}.
\begin{thm}\label{prop:afpol2af}
  Let $p, q$ be two distinct points of\/ $\fixaf$. Then
  \begin{ctext}
    $\bigcap \bigl\{ \{x\colon x\adjac y\}\colon y \adjac p, q \bigr\}$
  \end{ctext}
  is the line of\/ $\fixaf$ that passes through $p,q$.
  Consequently, the structure \fixaf\ is definable in terms of the binary collinearity of\/ 
  $\fixafpol$\ and thus it is definable in \fixafpol \ as well.
\end{thm}
\begin{proof}
  It suffices to note that $H_y = \{x\colon x\adjac y\}$ is a polar hyperplane
  of a point $y$ and if $y \adjac p, q$, then $p, q\in H_y$.
\end{proof}

\ifrach
\else
Interpretation of the isotropic (singular) subspaces of higher dimensions 
in $\fixafpol$ that makes use of the map $\rho$, analogous to
\ref{lem:anal:izolines}, remains, clearly, valid. An interested
reader can consider this problem as an easy exercise.
\fi


\section{Semiforms}\label{sec:semiforms}

The construction of the function $\rho$ in \eqref{def:metryka}
falls into the following more general one.

\begin{dfn}\label{def:cosik}
  Let $\field V$, ${\field V}'$ be vector spaces over a (commutative) field $\goth F$ with
  $\mathrm{char}({\goth F})\neq 2$. Let $V,V'$ be their sets of vectors and
  $\theta,\btheta$ be their zero-vectors, respectively.
\begin{sentences}\itemsep-2pt
\item\label{cosik:zal1}
  Let $\eta\colon V\times V\longrightarrow V'$ be an alternating bilinear map.
  Then $\eta(u_1,u_2) = - \eta(u_2,u_1)$ and $\eta(u,u) = \btheta$
  for all vectors $u,u_1,u_2 \in V$.
\item\label{cosik:zal2}
  Let $\delta\colon V'\times V' \longrightarrow V'$ be a map
  that satisfies the following conditions
  \begin{enumerate}[\sffamily\axC1.]\itemsep-2pt
  \item\label{delt4}
    $\delta(v_1+v,v_2+v) = \delta(v_1,v_2)$,
  \item\label{delt3}
    $\delta(\alpha v_1,\alpha v_2) = \alpha\delta(v_1,v_2)$,
  \item\label{delt2}
    $\delta(v_1,v) + \delta(v,v_2) = \delta(v_1,v_2)$.
  \end{enumerate}
  for all scalars $\alpha$ and $v,v_1,v_2 \in V'$.
\end{sentences}
  Set $Y := V'\times V$ and ${\field Y} := {\field V}'\oplus{\field V}$.
  On $Y$ we define the binary operation 
  $\semform\colon Y\times Y \longrightarrow V'$ by the formula
  \begin{equation}\label{def:eta+delta->rho}
    \semform([v_1,u_1],[v_2,u_2]) := \eta(u_1,u_2) - \delta(v_1,v_2).
  \end{equation}
  The resulting map $\semform$ is referred to as a {\em semiform} defined on $\field Y$.
\end{dfn}
An alternating bilinear form $\eta$ considered in \ref{def:cosik} is {\em nondegenerate}
when for each $\theta\neq u_1\in V$ there is $u_2\in V$ such that
$\eta(u_1,u_2)\neq \btheta$.

The following technical but important formulas are immediate from definition.
Let $p_i=[v_i,u_i]$, $q = [v,y]$.
\begin{eqnarray}
  \label{eq:imforA}
  \semform(\alpha p_1,\alpha p_2) - \alpha\semform(p_1,p_2) & = 
  & \alpha(\alpha-1)\eta(u_1,u_2);
  \\
  \label{eq:imforB}
  \semform(p_1+q,p_2+q) - \semform(p_1,p_2) & = & \eta(u_1-u_2,y);
  \\ \nonumber
  \text{in particular,}
  \\
  \label{eq:imforC}
  \semform(p_1,p_1+p_2) - \semform(\theta,p_2) & = & \eta(u_1,u_2).
  \\
  \label{eq:imforD}
  \semform(q,\theta) & = & v,
  \\
  \label{eq:imforE}
  \semform(\alpha p_1,q) - \alpha\semform(p_1,q) & = & (1-\alpha) v,
  \\
  \label{eq:imforF}
  \semform(p_1 + p_2,q) - (\semform(p_1,q)+\semform(p_2,q)) & = & -v.
\end{eqnarray}

One example is crucial:
\begin{exm}\label{exm:syaps}
  Let $\eta$ be a null-form defined on $\field V$. 
  Next, let 
    ${\field V}' = {\goth F}$ and $\delta(a,b) = a - b$.
  Definition \ref{def:cosik} coincides with the definition of the function
  $\rho$ in \eqref{def:metryka}.
\end{exm}

\begin{exm}\label{exm:power}
  Each alternating map 
  $\eta\colon V\times V \longrightarrow V'$  
  is derived from a linear map
  $g\colon \bigwedge^2{\field V} \longrightarrow V'$
  by the formula
\begin{equation}\label{eq:hom2form}
  \eta(u_1,u_2) = g(u_1 \wedge u_2).
\end{equation}
(see any standard textbook, e.g. \cite[{Ch. XIX}]{langproper}).

  We shall write, generally, (cf. \eqref{fixvar}) 
  $\eta_u$ for the map defined by $\eta_u(v) = \eta(u,v)$.
  It is a folklore that 
  $\dim(\bigwedge^2{\field V}) = \binom{n}{2}$, where $n = \dim({\field V})$.
  Note that when $u$ is fixed then the set 
  $S_u := \{ u\wedge y\colon y \in V \} = \Img(\wedge_u)$
  is a $(n-1)$-dimensional vector subspace of $\bigwedge^2{\field V}$.

  Clearly, the operation $\eta = \wedge$ together with a given $\delta$
  determines via \eqref{def:eta+delta->rho}
  a semiform.
\end{exm}
\begin{exm}\label{exm:prod}
  Let $\field V$ be a $3$-dimensional vector space.
  Then $\bigwedge^2{\field V} \cong {\field V}$
  and we can write $u' \wedge u'' = u'\vecprod u''$, where
  $\vecprod\colon V\times V\longrightarrow V$
  is a vector product defined on $\field V$.
  A standard formula defining $\vecprod$ 
  is the following:
  \begin{equation*}
    [\alpha'_1,\alpha'_2,\alpha'_3]\vecprod
    [\alpha''_1,\alpha''_2,\alpha''_3] =
    \left[
    \varepsilon_1
    \left|\begin{array}{ll}
      \alpha'_2 & \alpha'_3 \\ \alpha''_2 & \alpha''_3
    \end{array}\right|,
    \varepsilon_2
    \left|\begin{array}{ll}
      \alpha'_1 & \alpha'_3 \\ \alpha''_1 & \alpha''_3
    \end{array}\right|,
    \varepsilon_3
    \left|\begin{array}{ll}
      \alpha'_1 & \alpha'_2 \\ \alpha''_1 & \alpha''_2
    \end{array}\right|
    \right] 
  \end{equation*}
  with $\varepsilon_i = \pm 1$ (cf. \cite{mafodja:dim1}, \cite{lang}).
  Then $\semform$ defined on ${\field V}\oplus{\field V}$ by the formula
  \begin{equation*}
    \semform([v_1,u_1],[v_2,u_2]) = u_1 \vecprod u_2 - (v_1 - v_2)
  \end{equation*}
  is a semiform.
\end{exm}


\subsection{Affine atlas and its characterization}

\ifrach\else
In this and the forthcoming subsections \ref{ssec:synthetic} and \ref{ssec:canonical}
most of the proofs consist in direct 
computations and therefore they are left for the reader.
\fi

Let us give a more explicit representation of a map $\delta$ characterized
in \ref{def:cosik}\eqref{cosik:zal2}.
\begin{lem}\label{lem1:delt}
  Let $\delta$ meet conditions 
  {\refaxC{delt4}-\refaxC{delt2}} of {\upshape\ref{def:cosik}\eqref{cosik:zal2}}.
  Then the following conditions follow as well:
    \begin{enumerate}[\upshape\sffamily\axC1.]\setcounter{enumi}{3}\itemsep-2pt
    \item\label{delt0}
      $\delta(\btheta,\btheta) = \btheta$ (by \refaxC{delt3});
    \item\label{delt00}
      $\delta(v,v) = \btheta$ (by \refaxC{delt4}, \refaxC{delt0});
    \item\label{delt1}
      $\delta(v_1,v_2) = - \delta(v_2,v_1)$ (by \refaxC{delt2}, \refaxC{delt00});
    \item\label{delt5}
      $\delta(v_1+v_2,\btheta) = \delta(v_1,\btheta) + \delta(v_2,\btheta)$ 
      (by \refaxC{delt4} -- \refaxC{delt2}, \refaxC{delt1});
    \end{enumerate}
    for all $v,v_1,v_2 \in V'$.
  \par
  Define $\sematlas\colon V'\longrightarrow V'$ by the formula
  $\sematlas(v) = \delta(v,\btheta)$.
  Then $\sematlas$ is a linear map and $\delta$ is characterized by the formula
  \begin{equation}\label{wzornadelt}
    \delta(v_1,v_2) = \sematlas(v_1) - \sematlas(v_2)\;\; \text{(}=\sematlas(v_1-v_2)\text{)}.
  \end{equation}
\end{lem}
A map $\delta$ defined by formula \eqref{wzornadelt} is called
{\em an affine atlas}, it is
{\em nondegenerate} when $\sematlas$ is an injection (i.e. if $\ker(\sematlas)$ is trivial).
Note that when $\dim({\field V}')<\infty$ and $\delta$ is nondegenerate
then the representing map $\sematlas$ is a surjection as well.

The following is straightforward
\begin{lem}\label{lem2:delt}
  Let $\sematlas\colon V'\longrightarrow V'$ be a linear map
  and $\delta$ be defined by \eqref{wzornadelt}.
  Then $\delta$ meets conditions 
  {\refaxC{delt4}-\refaxC{delt2}} of {\upshape \ref{def:cosik}\eqref{cosik:zal2}}.
\end{lem}
Finally, we note that affine atlases 
can be equivalently
characterized by another, less elegant but more convenient for our 
further characterizations, set of postulates.
\begin{lem}\label{lem3:delt}
  Let $\delta$ satisfy the postulates 
  \refaxC{delt3}, \refaxC{delt5}, \refaxC{delt1}, \refaxC{delt4}
  of {\upshape\ref{def:cosik}\eqref{cosik:zal2}}, {\upshape\ref{lem1:delt}}.
  Then $\delta$ satisfies \refaxC{delt2} as well.
\end{lem}


\subsection{Synthetic characterization and representations of semiforms}\label{ssec:synthetic}

Let 
${\field Y} = \struct{Y,+,\theta}$, 
${\field Z} = \struct{Z,+,\btheta}$ 
be vector spaces with the common field
$\goth F$ of scalars.
Let $\semform\colon  Y\times Y \longrightarrow Z$ 
be a map.
Consider the following properties:
\begin{enumerate}[\upshape\sffamily\axA1.]\itemsep-2pt
\item\label{ax1}
  $\semform(p,q) = - \semform(q,p)$ for each $p,q \in Y$.
\item\label{ax2}
  If $\semform(\theta,p) = \btheta$ then $\semform(\alpha q,p) = \alpha\semform(q,p)$
  for each scalar $\alpha$ and each vector $q$.
\item\label{ax3}
  If $\semform(\theta,p) = \btheta$ then 
  $\semform(q_1+q_2,p) = \semform(q_1,p)+\semform(q_2,p)$.
\item\label{ax2-3}
  If $p\neq \theta$ then there is $q$ with $\semform(p,q)\neq \btheta$
  and $\semform(\theta,q)=\btheta$.
\item\label{ax4-5}
  $\semform(-p,-q) + \semform(p,q) = 2(\semform(p,p+q) - \semform(\theta,q))$.
\item\label{ax4}
  $(\forall p)[\semform(p+q,q) = \semform(p,\theta)]$ implies
  $(\forall p_1,p_2)[\semform(p_1+q,p_2+q) = \semform(p_1,p_2)]$.
\item\label{ax5}
  $2 \big(\semform(\alpha p_1,\alpha p_2) - \alpha\semform(p_1,p_2) \big) =
  \alpha(\alpha-1)\big( \semform(-p_1,-p_2) + \semform(p_1,p_2) \big)$.
\item\label{ax7}
  For each $q\in Y$ there is $p\in Y$ such that $\semform(p,\theta)=\btheta$ and
  $\semform(p-q,-r) = -\semform(q-p,r)$ for all $r\in Y$.
\end{enumerate}

In view of formulas \eqref{eq:imforA} -- \eqref{eq:imforF} it is evident that Axioms \refaxA{ax1} -- \refaxA{ax7}
are satisfied by each semiform as defined in \eqref{def:eta+delta->rho}.

  Set 
    $M := \{ p\in Y\colon \semform(\theta,p)=\btheta \}$.
  With each $p\in Y$ we associate the map 
\begin{equation}\label{fixvar}  
  \semform_p\colon Y\longrightarrow Z,\quad\quad \semform_p(q) = \semform(q,p).
\end{equation}
Note that if $\semform$ is a semiform defined in \ref{def:cosik} then
$M = V$ and $\semform\restriction{M\times M} = \eta$.

Recall that in one of the most intensively investigated cases in geometry
when we consider a sesquilinear form  $\semform$, 
$M = Y$, 
$\semform_p$ is a linear map, and
$p\mapsto\semform_p$ is semilinear. 
Our axioms  lead to a similar situation.

\begin{lem}\label{lem:ax1}
 If $\semform$ satisfies \refaxA{ax1} then $\semform(p,p) = \btheta$ for each $p \in Y$.
 Consequently, 
 \begin{ctext}
   $\theta\in M$.
 \end{ctext}
\end{lem}
\begin{lem}\label{lem:ax2+3}
  Assume Axiom \refaxA{ax1}. Let $p \in Y$.
  \begin{sentences}\itemsep-2pt
  \item\label{ax3x}
    If $\semform_p$ is additive then $p\in M$.
  \item\label{ax2x}
    If $\semform_p$ is multiplicative then $p \in M$.
  \end{sentences}
  Consequently, if Axioms \refaxA{ax2} and \refaxA{ax3} are valid then 
  the map $\semform_p$ is linear iff $p\in M$.
  \par
  In particular (cf. \ref{lem:ax1}), 
  $\semform_\theta$ is a linear map, i.e. the following hold:
  \begin{eqnarray*}
    \semform(\alpha p,\theta) & = & \alpha\semform(p,\theta),
    \\
    \semform(p_1+p_2,\theta)  & = & \semform(p_1,\theta) + \semform(p_2,\theta).
  \end{eqnarray*}
  Clearly, $M = \ker(\semform_\theta)$ and thus $M$ is a subspace of $\field Y$.
  \par
  If, moreover, Axiom \refaxA{ax2-3} is valid then the assignment
  $M\ni p\longmapsto\semform_p$ is injective.
\end{lem}
\begin{lem}\label{lem:ax4+5}
  \begin{sentences}\itemsep-2pt
  \item\label{ax4x}
    Set \hfill
      $D' := \{ q \in Y \colon (\forall p\in Y)[\semform(q,q+p) = \semform(\theta,p)] \}$.
    \hfill\strut
    \\
    Then $\theta \in D'$ and the set $D'$ is closed under vector addition.
  \item\label{ax4a}
    Assume Axiom \refaxA{ax4}. 
    Then $q \in D'$ iff the condition 
    $\semform(p_1+q,p_2+q) = \semform(p_1,p_2)$ holds for all $p_1,p_2\in Y$.
  \item\label{ax5x}
    Set \hfill
      $D'' := \{ q \in Y \colon (\forall p\in Y)[\semform(-p,-q) = - \semform(p,q)] \}$.
    \hfill\strut
    \\
    If Axiom \refaxA{ax5} is valid then the set $D''$ is closed under scalar multiplication.
  \end{sentences}
  If Axiom \refaxA{ax4-5} is adopted then $D' = D''$.
  \par
  Consequently, if Axioms \refaxA{ax4}, \refaxA{ax5}, and \refaxA{ax4-5} are 
  valid then $D:= D' =D''$ is a vector subspace of $\field Y$.
  \par
  Moreover, if Axioms \refaxA{ax1}-\refaxA{ax5} are valid then $M\cap D = \{\theta\}$.
\end{lem}
\begin{lem}\label{lem:rekonstrukcja}
  With the Axioms \refaxA{ax1}-\refaxA{ax5},  Axiom \refaxA{ax7} can be expressed as the 
  following statement:
  \begin{ctext}
    ${\field Y} = D\oplus M$.
  \end{ctext}
  Assume that the Axioms \refaxA{ax1}-\refaxA{ax7} are valid and set
  $\eta:= \semform\restriction{M\times M}$, $\delta:=\semform\restriction{D\times D}$.
  Then $\eta$ is an alternating nondegenerate vector-valued form.
  The map $\delta$ is a nondegenerate affine atlas; it is determined by
  a linear injection $\sematlas\colon D\longrightarrow D$  by the formula \eqref{wzornadelt}. 
\end{lem}
\begin{lem}\label{wzorek}
  Let $q_i = p_i + r_i$ with $p_i\in M$, $r_i\in D$ for $i =1,2$.
  Then, we have (cf. \eqref{def:eta+delta->rho}) the following
  \begin{equation}
    \semform(q_1,q_2) = \eta(p_1,p_2) - \delta(r_1,r_2).
  \end{equation}
\end{lem}
Summing up the above, with not too tedious computation, we close this
part by the following representation theorem
\begin{thm}
  Let $\semform\colon Y\times Y\longrightarrow Z$ be a map.
  The following conditions are equivalent.
  \begin{sentences}\itemsep-2pt
  \item
    $\semform$ is a semiform defined in accordance with {\upshape\ref{def:cosik}}, where $\eta,\delta$
    are nondegenerate.
  \item
    $\semform$ satisfies Axioms \refaxA{ax1}-\refaxA{ax7}.
  \end{sentences}
\end{thm}

\begin{rem*}
  A nondegenerate semiform $\semform$ is scalar valued 
  (i.e. $\dim(Z) = 1$) iff 
  {it is associated with} 
  a symplectic polar space.
\end{rem*}
\begin{exm*}
  Let $\dim({\field V})=2$. Then the determinant is a symplectic form.
  Therefore the map below ($x_i,y_i$ are 
  elements of the field of scalars $\goth F$ of $\field V$)
  \begin{equation*}
    \semform([x_1,x_2,x_3],[y_1,y_2,y_3]) = 
    \left| \begin{array}{ll} x_2 & x_3 \\ y_2 & y_3 \end{array}\right|
    - (x_1 - y_1)
  \end{equation*}
  is a semiform. The associated aps is determined by the so called line
  complex in the $3$-dimensional projective space over $\goth F$
  (cf. \cite[Ch. 6]{hart}, \cite[Vol. 2, Ch. 9, Sec. 3]{hodge-pedoe}).
\end{exm*}


\subsection{A simplification of semiforms}\label{ssec:canonical}

Forthcoming constructions are provided for a fixed nondegenerate semiform $\semform$
defined in \ref{def:cosik}.
Moreover, we assume that 
\begin{ctext}
  $\dim({\field V}') =: \dimcel < \infty$.
\end{ctext}
Set $\fixaf = \AfSpace({},{\field Y})$.
Let $p_1,p_2$ be vectors of $\field Y$, so 
$p_i = [v_i,u_i]$, $v_i \in V'$, $u_i \in V$.
By definition, 
  $$\semform(p_1,p_2) = \eta(u_1,u_2) - \sematlas(v_1-v_2)$$
for suitable maps $\eta$, $\sematlas$ 
(recall, they need to be nondegenerate. As a consequence, $\sematlas\in GL({\field V}')$).

  There is, generally, a great variety of semiforms. But some of them may lead to 
  isomorphic geometries.
  Write $\semform_{\eta,\phi}$ for $\semform$ defined by \ref{def:cosik}
  with $\delta$ defined by \eqref{wzornadelt}.
  We have evident
  \begin{prop}\label{prop:uzmiennic}\strut
    \begin{sentences}
    \item\label{uzmiennic:war1}
      There is a linear bijection $\Phi\in GL({\field Y})$ such that for 
      any $q_1,q_2 \in Y$ it holds:
      \begin{ctext}
        $\semform_{\eta,\phi}(q_1,q_2) = \semform_{\eta,\id} (\Phi(q_1),\Phi(q_2))$
      \end{ctext}
    \item\label{uzmiennic:war2}
      Let $B \in GL({\field V})$, $\gamma$ be a non zero scalar. Then, clearly,
      the map $\gamma\eta B$ defined by 
      $\gamma\eta B(u_1,u_2) = \gamma\cdot \eta(B(u_1),B(u_2))$ is an alternating form.
      There is a linear bijection $\Phi\in GL({\field V})$ such that
      the following holds for any $q_1,q_2 \in Y$
      \begin{ctext}
        $\semform_{\gamma\eta B,\id}(q_1,q_2) = 
        \gamma^{-1}\cdot\semform_{\eta,\id}(\Phi(q_1),\Phi,(q_2))$.
      \end{ctext}
    \end{sentences}
  \end{prop}
\begin{rem*}
  In terms of {\upshape\ref{exm:power}} we have 
  $\eta B  = g \circ (B \wedge B)$.
\end{rem*}
In view of \ref{prop:uzmiennic}, till the end of our paper we assume that
$\semform$ is defined by the formula of the form
\begin{equation}\label{eq:semform:spec}
  \semform([v_1,u_1],[v_2,u_2]) = \eta(u_1,u_2) - (v_1 - v_2).
\end{equation}


\subsection{\expandafter\MakeUppercase\afsempol s}\label{ssec:afsempol}

Imitating \ref{lem:anal:colin}, for points $p_1, p_2$ of $\fixaf$, we put generally
\begin{equation}\label{eq:badjac}
  p_1 \badjac p_1 \iff \semform(p_1,p_2) = \btheta.
\end{equation}
From definition it is immediate that
$p_1 \badjac p_2 \iff \eta(u_1,u_2) = v_1-v_2$.
\begin{lem}\label{lem:lines:adjclosed}
  Let $p_1,p_2$ be two distinct points of \fixaf\ and $L = \LineOn(p_1,p_2)$.
  If $p_1\badjac p_2$ then $q_1 \badjac q_2$ for all $q_1,q_2 \in L$.
\end{lem}
\begin{proof}
  Let us write $p_1 = [v,u]$, $p_2 = p_1 + q$, and  $q = [v_0,u_0]$.
  From assumption,
  $\eta(u,u+u_0) = v-(v+v_0) = - v_0$.
  We directly compute that then
  $\eta(u+\alpha u_0,u+\beta u_0) = (\alpha - \beta)\eta(u,u_0) =
  (\alpha-\beta) v_0 = (v+\alpha v_0)-(v+\beta v_0)$.
  This yields $p_1 + \alpha q \badjac p_1 + \beta q$
  for any scalars $\alpha,\beta$ and closes the proof.
\end{proof}
In view of \ref{lem:lines:adjclosed}, the relation $\badjac$ determines
the class $\izolines$ of lines of \fixaf\ by the condition
\begin{multline}\label{eq:izolines}
  L \in \izolines \text{ iff } p_1 \badjac p_2 \text{ for any } p_1,p_2\in L; \\
    \text{ equivalently, iff } p_1 \badjac p_2 \text{ for a pair } p_1,p_2 \text{ of distinct points on } L.
\end{multline}

For computation it is convenient to have this criteria (comp. \ref{lem:anal:izolines}):
\begin{equation}\label{eq:dirlineizo}
  [v_0, u_0]+\gen{[v, u]}\in\izolines\text{ iff } \eta(u_0, u) = -v.
\end{equation}
\ifrach
\else
It is a straightfoward consequence of \eqref{eq:izolines} and \eqref{eq:badjac}.
\fi

The class $\izolines$ induces the incidence structure  $\struct{Y, \izolines}$ 
that we will take a look into. Let us call this structure 
the {\em \afsempol}\ determined by $\semform$.
\begin{lem}\label{lem:propG}\strut
  \begin{sentences}
  \item\label{lem:propG:parallel}
  The class $\izolines$ is unclosed under parallelism, i.e. 
  for every $L_1\in\izolines$ there is an affine line $L_2\not\in\izolines$
  such that $L_1\parallel L_2$.
  \item\label{lem:propG::alaShult}
  Let $L_1, L_2\in\izolines$, $L_1\neq L_2$, and $p\in L_1\cap L_2$. If
  $L$ is an affine line through $p$ from the affine plane $\gen{L_1, L_2}$, then $L\in\izolines$.
  \end{sentences}
\end{lem}

\begin{proof}
  \eqref{lem:propG:parallel}:
  \ifrach
\else
    Straightforward computation.  
  \fi

  \eqref{lem:propG::alaShult}:
  Without loss of generality we can assume that $p=[v_0,u_0]$ and $L_i = p+\gen{a_i}$
  where $a_i\in Y$, $i=1,2$.
  Then $L = p + \gen{\alpha_1a_1 +\alpha_2a_2}$ for some 
  $\alpha_i\in F$. Applying \eqref{eq:dirlineizo} to $L_1, L_2$ and then to $L$
  we are through.
\end{proof}

\begin{thm}\label{thm:Gamma}
  The \afsempol\ 
  determined by a semiform
  is a $\Gamma$-space and
  its every singular subspace carries affine geometry.
\end{thm}
\begin{proof}
  Let $\struct{Y,\izolines}$ be our \afsempol.
  The first part  follows directly from \ref{lem:propG}\eqref{lem:propG::alaShult}.
  The other part is a simple observation that a singular subspace of $\struct{Y, \izolines}$,
  in other words, a strong subspace wrt. $\badjac$ in $\fixaf$, is an affine subspace 
  of $\fixaf$.
\end{proof}

When we deal with a $\Gamma$-space a  question on the form of its triangles
may appear important. The following is immediate from \eqref{eq:dirlineizo} and
\eqref{eq:badjac}.
\begin{rem}\label{rem:triangle}
  A triangle in an \afsempol\ $\struct{Y,\izolines}$ has form
  \begin{equation}\label{eq:triangle}
    [v_0,u_0], \qquad
    [v_0 + \eta(u,u_0),u_0 + u], \qquad
    [v_0 + \eta(y,u_0), u_0 + y],
  \end{equation}
  where $\eta(u,y) = \btheta$.
\end{rem}
\begin{cor}\label{cor:beztrg}
  If $\dim(\ker(\eta_u)) = 1$ for each nonzero vector $u \in V$ then 
  the corresponding \afsempol\
  contains no proper triangle.
  In that case its maximal singular subspaces are the lines.
\end{cor}
\begin{exm}\label{exm:siatki}
  In view of \ref{thm:Gamma} one could expect that {\afsempol s} are models
  of the system considered in \cite{cyup-pas}.
  \par
  In case considered in \ref{exm:power} and, consequently, in case considered
  in \ref{exm:prod} we have $\dim(\ker(\eta_u)) = 1$ for all $u\neq \theta$ and 
  corresponding alternating map $\eta$. 
  Therefore, {\em the structure $\struct{Y,\izolines}$ has no triangles}.
  So, {\em {\afsempol s} determined by them
  are not models of the system considered in \cite{cyup-pas}}.
  \par
  One can also compute that, e.g. an \afsempol\ determined by the exterior
  power operation {\em is not a generalized quadrangle}.
\end{exm}

In the sequel we shall frequently consider the condition (with prescribed values $u,v$)
\begin{ctext}
  $(\exists y)\;[\; \eta(u,y) = v \;]$;
\end{ctext}
applying the representation given in \ref{exm:power}
this can be read as $(\exists y)\;[\; g(u\wedge y) = v \;]$,
which is equivalent to 
$(\exists \omega \in S_u)\;[\; g(\omega) = v \;]$.
This observation allows us to construct quite ``strange'' (`locally surjective') 
alternating maps.

As an immediate consequence of \ref{lem:lines:adjclosed} and the definition we have
\begin{lem}\label{lem:pomo0G}
  Let $q = [v_0,u_0]$ be a vector of $\field Y$. 
  The following conditions are equivalent.
  \begin{sentences}\itemsep-2pt
  \item
    There is no line $L\in\izolines$ with the direction $q$.
  \item
    The equation
    \begin{equation}\label{eq:kiedyna}
      \eta(u_0,u) = v_0 
    \end{equation}
    is not solvable in $u$.
  \end{sentences}
  \par
  In particular, if $u_0 = \theta$ and $v_0 \neq \btheta$
  then \eqref{eq:kiedyna} is not solvable and thus there is no line $L\in\izolines$
  with the direction $q$.
\end{lem}
Set 
\begin{equation*}
  D := \{ q\in Y \colon \text{ no line in }\izolines \text{ has the direction } q\}
\end{equation*}
Note, in particular, that when $\eta_u\colon V \longrightarrow V'$ is a surjection
for each non zero vector $u$ then 
$D = V' \times \{ \theta \}$.
\begin{exmc}{exm:prod}\label{cont1}
  Let $\vecprod$ be a vector product in a vector $3$-space $\field V$ associated with
  a nondegenerate bilinear symmetric form $\xi$ and $\perp = \perp_\xi$
  be the orthogonality determined by $\xi$.
  Then for $u_0,v_0\neq\theta$ equation \eqref{eq:kiedyna} 
  is solvable  iff $u_0 \perp v_0$.
  In that case we have
  \begin{equation*}
    D = V \times \{ \theta \} \cup \{ [v,u]\in V\times V\colon u \not\perp v \}.
  \end{equation*}
\end{exmc}

\begin{lem}\label{lem:pomo1X}
  For a fixed $u_0\in V$, $v_0\in V'$ and  a scalar $\alpha$ the set 
  \begin{equation}\label{eq:hipciaX}
    {\cal Z} = \big\{ [v,u]\colon \eta(u_0,u) = v_0 + \alpha v \big\}
  \end{equation}	
  is a subspace of\/ \fixaf.
  The class of sets of form \eqref{eq:hipciaX} is invariant under
  translations of\/ \fixaf.
\end{lem}
\begin{proof}
  Take $[v_1,u_1],[v_2,u_2]\in {\cal Z}$ and an arbitrary scalar $\lambda$.
  Then we compute
\begin{math}
  \eta(u_0,\lambda u_1 +(1-\lambda) u_2) = \lambda\eta(u_0,u_1) + (1-\lambda)\eta(u_0,u_2)
  = \lambda(v_0 + \alpha v_1) + (1-\lambda)(v_0 + \alpha v_2) =
  v_0 + \big( \lambda v_1 + (1 - \lambda) v_2 \big)
\end{math}
  which proves that 
  $[\lambda v_1 + (1-\lambda) v_2,\lambda u_1 +(1-\lambda) u_2]\in {\cal Z}$
  and thus $\LineOn({[v_1,u_1]},{[v_2,u_2]}) \subset {\cal Z}$.
  This proves that ${\cal Z}$ is a subspace of \fixaf.

\par
  Write ${\cal Z}_{u_0,v_0,\alpha}$ for the set defined by \eqref{eq:hipciaX}.
  Let $q = [x,y]\in Y$ be arbitrary.
  Then
  \begin{math}
    \tau_{q}([v,u]) = 
    [v+x,u+y] \in {\cal Z}_{u_0,v_0,\alpha} \iff
    \eta(u_0,u+y) = v_0 + \alpha (v+x) \iff
    \eta(u_0,u) = (v_0 - \eta(u_0,y) + \alpha x) + \alpha v \iff
    [v,u] \in {\cal Z}_{u_0,v_0 - \eta(u_0,y) + \alpha x,\alpha}.
  \end{math}
  Thus 
    $$\tau_{q}^{-1}({\cal Z}_{u_0,v_0,\alpha}) = 
      {\cal Z}_{u_0,v_0 - \eta(u_0,y) + \alpha x,\alpha}.$$
  This closes our proof.
\end{proof}

\begin{lem}\label{lem:pomo1Y}
  Let   $\cal Z$ be defined by \eqref{eq:hipciaX}.
  Then either $\cal Z$ is an empty set or it is an affine subspace of\/ $\fixaf$
  with the dimension 
  $\dimcel + \dim(\ker(\eta_u))$, or with the dimension $\dim(\field V)$.
\end{lem}
\begin{proof}
  If $\cal Z$ is nonempty then by \ref{lem:pomo1X} we can assume that 
  $[\btheta,\theta]\in{\cal Z}$. Then $\cal Z$ is characterized by an
  equation $\eta(u_0,u) = \alpha_0 v$ with prescribed values of $u_0$, $\alpha_0$
  and it is the kernel of the linear map $\Psi\colon Y \longrightarrow V'$,
  $\Psi [v,u]\longmapsto \eta(u_0,u) - \alpha_0 v$.
  If $\alpha_0 = 0$ then, clearly, 
  ${\cal Z} = V'\times \ker(\eta_{u_0})$ and thus
  $\dim({\cal Z}) = \dim(\ker(\eta_{u_0}) + \dim({\field V}') = \dimcel + \dim(\ker(\eta_{u_0}))$).
  Assume that $\alpha_0 \neq 0$; then $\cal Z$ can be considered as the kernel
  of the map $[v,u]\longmapsto \eta(\frac{1}{\alpha}u_0,u) - v$.
  Let $(d_1,...,d_k)$ be a linear basis of $\Img(\eta_{u_0})$ and 
  $(e_1,...,e_m)$ be a basis of $\ker(\eta_{u_0})$.
  Choose one $z_i \in V$ with $\eta_{u_0}(z_i) = d_i$ for each $i=1,...,k$;
  Then the set $\{ z_1,...,z_k \}$ is linearly independent.
  Moreover, the subspaces $\gen{z_1,...,z_k}$ and $\ker(\eta_{u_0})$ have only
  the zero vector in common.
  A basis of $\cal Z$ consists of the vectors
  $$
  \left(  
  [d_1,z_1],[d_1,z_1+e_1],\ldots,[d_1,z_1+e_m], [d_2,z_2],\ldots,[d_k,z_k]
  \right).
  $$
  Consequently, $\dim({\cal Z}) = \dim(\ker(\eta_{u_0})) + \dim(\Img(\eta_{u_0})) = 
  \dim(\Dom(\eta_{u_0})) = \dim({\field V})$.
\end{proof}
Applying \ref{lem:pomo1Y} and \eqref{eq:badjac} we get, e.g. a geometrically
interesting strengthening of \ref{lem:propG}\eqref{lem:propG::alaShult}:
\begin{cor}
  The set of points that are joinable in an \afsempol\  
  (determined by a semiform defined on $\field Y$)
  with a given point
  is a subspace of dimension $\dim({\field V})$ in the surrounding affine space \fixaf.
\end{cor}

Another corollary of analogous type that will appear important in the sequel is read
as follows.
\begin{lem}\label{lem:izoline=przekroj}
  Assume the following: 
  \begin{enumerate}[\axE]
  \item
  If $u'\nparallel u''$ are two vectors of $\field V$
  then there is $y_0$ such that $\eta(u',y_0) = \btheta$ and $\eta(u'',y_0)\neq\btheta$.
  \end{enumerate}
  Let $L\in\izolines$ pass through $p = [\btheta,\theta]$ and $p'$ be a point on $L$.
  Then $L = \bigcap\big\{ [q]_{\badjac} \colon q \badjac p,p'  \big\}$.
\end{lem}

Finally, let us make a few comments that enable us to characterize 
(with the help of \ref{rem:triangle}) 
the geometry of the lines and the planes through a point in an \afsempol.
\par
Each alternating map $\eta \colon V \times V \longrightarrow V'$
determines the incidence substructure 
$\KwadrSpace(\eta,{\field V})$
of the projective space $\PencSpace({},{\field V})$
with the point set 
unchanged
and with the class $\lines^\ast$ of projective lines of the form
$\gen{u',u''}$, where $u',u'\in V$ are linearly independent and 
$\eta(u',u'') = \btheta$ as its lines.
With a fixed basis of ${\field V}'$ one can write $\eta$ as the
(Cartesian) product of $\dimcel$ bilinear alternating forms 
$\eta_i\colon V\times V\longrightarrow F$:
\begin{equation}
  \eta(u',u'') = [\eta_1(u',u''),\ldots,\eta_{\dimcel}(u',u'')],
\end{equation}
clearly, the $\eta_i$ need not be nondegenerate.
So, each $\eta_i$ determines a (possibly degenerate) null system
$\KwadrSpace(\eta_i,{\field V})$ with the lines $\Quadr_2(\eta_i)$.
The class $\lines^\ast$ is simply $\bigcap_{i=1}^{\dimcel} \Quadr_2(\eta_i)$.
\begin{prop}\label{lem:geowiazki:0}
  The geometry of the lines and planes of an \afsempol\ 
  (determined by a semiform $\semform$ associated via \eqref{eq:semform:spec}
  with an alternating map $\eta$)
  which pass through the 
  point $[\btheta,\theta]$ is isomorphic to $\KwadrSpace(\eta,{\field V})$.
\end{prop}
\begin{proof}
  Let $p = [\btheta,\theta]$.
  In view of \eqref{eq:dirlineizo} the class of lines through $p$ is the set
  $\{\gen{[\btheta,u]} \colon u \text{ is a nonzero vector}\}$ so, it can be identified
  with the point set of $\PencSpace({},{\field V})$ under the map
  $\gen{[\btheta,u]}\mapsto\gen{u}$. 
  From \ref{rem:triangle} we infer that two lines $\gen{[\btheta,u']}$,
  $\gen{[\btheta,u'']}$ span a plane in the corresponding \afsempol\ iff 
  $\eta(u',u'') = \btheta$, which closes our reasoning.
\end{proof}
From the homogeneity of each \afsempol\ 
(which will be proved later in \ref{prop:transtiv:GEN})
one will get that the geometry of the lines and the planes through arbitrary point
of an \afsempol\ is (in the above sense) a generalized null system.


\subsection{Automorphisms}\label{ssec:aut}

To establish the automorphism group of the relation  $\badjac$ we need 
some additional assumptions. One of these conditions  is read 
as follows:

\begin{enumerate}[\axD]
\item\label{eq:cond}
The set of directions of $V'\times\{\theta\}$ can be characterized
in terms of the projective geometry of the horizon of $\AfSpace({}, {\field Y})$
with the set of directions of $D$ distinguished.
\end{enumerate}
Clearly, in view of \ref{lem:pomo0G} this condition holds when $\semform$
is scalar valued.
Let us point out that it is not a unique possibility when this condition
is valid.

\begin{exmc}{exm:prod}\label{cont2}
  We continue with the notation of \ref{exm:prod}. 
  \emph{Let $f\in\Gamma L(Y)$ preserve the set of directions $D$. Then $f$
  preserves the vector subspace $V\times\{\theta\}$.}
  Indeed, the geometric structure of the complement of $D$  carries 
  the geometry of the reduct ${\goth R}({\field V}, \xi)$ of 
  a hyperbolic polar space of the form  considered in \ref{sssec:hyperbolic}.
  Our claim follows from \ref{thm:hippolreduct}.
  Consequently, the condition \axD\ is valid here.
\end{exmc}

\begin{prop}\label{prop:aut:GEN}
  If $F$ is given by the formula
  \begin{equation}\label{wz:aut:GEN}
    F([v,u]) = [\psi_1(v) + \psi_2(u) + v_0, \varphi(u) + u_0]
  \end{equation}
  where
  $v_0 \in V'$, $u_0 \in V$,
  $\psi_1\colon V'\longrightarrow V'$,
  $\varphi\colon V\longrightarrow V$
  are linear bijections,
  $\psi_2\colon V\longrightarrow V'$,
  and the following holds:
  \begin{enumerate}[a)]\itemsep-2pt
  \item\label{wrr1G}
    $\psi_2( u ) = 
    \eta\bigl(\varphi(u),u_0\bigr)$ for every vector $u$ of\/ $\field V$, and
  \item\label{wrr2G}
    $\eta(\varphi(u_1),\varphi(u_2)) = \psi_1\eta(u_1,u_2)$, 
    for all vectors $u_1,u_2$ of\/ $\field V$,
  \end{enumerate}
  then $F$ preserves the relation $\badjac$.
  In that case the semiform $\semform$ is transformed under the rule
  \begin{equation}\label{eq:zmianaformy}
    \semform\bigl(F(p_1),F(p_2)\bigr) = \psi_1\bigl(\semform(p_1,p_2)\bigr)
  \end{equation}
  for any pair $p_1,p_2$ of points of\/ \fixaf.

  Conversely, under additional assumption that \axD\ is valid, each 
  linear (affine) automorphism of\/ $\fixaf$ is of the form \eqref{wz:aut:GEN}.
\end{prop}
\begin{note*}\small
  If $\eta$ is `onto' $V'$ then
  given map $\varphi$, condition \rref{wrr2G} uniquely determines
  $\psi_1$.
  Similarly, for a given map $\varphi$ and vector $u_0$, 
  condition \rref{wrr1G} uniquely determines $\psi_2$.
\end{note*}
\begin{proof}
  Assume that $F$ is defined by the formula \eqref{wz:aut:GEN} and \rref{wrr1G},
  \rref{wrr2G} hold. Let $p_i = [v_i,u_i]$, $v_i \in V'$, $u_i \in V$, for $i=1,2$.
  We compute as follows:
  \begin{math}
    \semform(F(p_1),F(p_2)) =
    \eta(\varphi(u_1),\varphi(u_2)) + 
    \eta(\varphi(u_1-u_2),u_0) - \psi_1((v_1 - v_2)
    - \psi_2(u_1-u_2) =
    \psi_1\eta(u_1,u_2) + 
    \eta(\varphi(u_1-u_2),u_0) - \psi_1(v_1 - v_2) 
    - \eta(\varphi(u_1-u_2), u_0) =
    \psi_1\eta(u_1,u_2) - \psi_1(v_1 - v_2) =
    \psi_1(\semform(p_1,p_2)),
  \end{math}
  which proves \eqref{eq:zmianaformy}.
  This yields, in particular, that $F$ preserves $\badjac$.
 \par
  Now assume that $F$ is an affine automorphism of $\fixaf$ preserving $\badjac$
  and that \axD\ is valid. Then, $F$
  is a composition $\tau_{[v_0,u_0]} \circ F_0$,
  where $F_0\in\GL({\field Y})$ and  $[v_0,u_0]$ is a vector of $\field Y$.
  The map $F_0$ can be presented in the form 
  $F_0([v,u]) = [\psi_1(v) + \psi_2(u), \varphi_1(u) +\varphi_2(v)]$
  for suitable linear maps ($\varphi_2\colon V'\longrightarrow V$).
  By \ref{lem:pomo0G}, 
  the linear part $F_0$ of $F$ fixes the subspace $V'$ 
  and thus $\varphi_2 \equiv \theta$. We write $\varphi = \varphi_1$.
  Since $F$ preserves the relation $\badjac$ by definition
  we obtain  the following equivalence:
  \begin{multline}\label{wz:pres:GEN}
    v_1 - v_2 = \eta(u_1,u_2) \iff 
    \\
    \psi_1(v_1 - v_2) + \psi_2(u_1 - u_2) = 
    \eta\bigl(\varphi(u_1),\varphi(u_2)\bigr) + \eta\bigl(\varphi(u_1-u_2),u_0\bigr)
  \end{multline}
  for all vectors $v_1,v_2\in V'$, $u_1,u_2 \in  V$.
  Substituting in \eqref{wz:pres:GEN} $u_2 = \theta$ and $v_1 = v_2$ we arrive to 
  the condition \rref{wrr1G}.
  In particular, from \rref{wrr1G} we obtain 
  $\psi_2(u_1 - u_2) = \eta(\varphi(u_1-u_2),u_0)$ 
  for all $u_1,u_2$ in $\field V$.
  Thus, assuming \rref{wrr1G} from \eqref{wz:pres:GEN} we get \rref{wrr2G}.
  Finally, $F$ has form \eqref{wz:aut:GEN}, as required.
\end{proof}

\begin{prop}\label{prop:transtiv:GEN}
  The group of automorphisms of $\badjac$ is transitive.
\end{prop}
\begin{proof}
  It suffices to compute the orbit $\cal O$ of the point $[\btheta,\theta]$
  under the group of affine automorphisms of $\badjac$.
  From \ref{prop:aut:GEN}, the orbit $\cal O$ contains all the vectors
  $[\psi_2(\theta) + v_0,\varphi(\btheta)+u_0] = [v_0,u_0]$
  with suitable maps $\psi_2$, $\varphi$.
  Considering $\varphi = \id$, $\psi_2(u) = \eta(\varphi(u),u_0)$,
  $\psi_1= \id$ we get a class of affine automorphisms of $\badjac$:
  those defined by the formula
    $$ F ([v,u]) = [v+ \eta(u,u_0) + v_0, u + u_0]$$
  with arbitrary fixed $u_0,v_0$.
  So, each point of \fixaf\ is in $\cal O$.
\end{proof}

Combining \ref{prop:transtiv:GEN} and \ref{lem:izoline=przekroj}
we get a theorem, which is important in the context of foundations
of geometry of {\afsempol s}.
\begin{thm}\label{thm:izoline=przekroj}
  Let\/ $\goth B$ be the \afsempol\ determined by a semiform that meets assumptions \axE\
  of\/ \upshape{\ref{lem:izoline=przekroj}}. 
  For each pair $p,q$ of points of\/ $\goth B$ such that $p \badjac q$ the set
  \begin{equation}\label{eq:izoline=przekroj}
    \bigcap \bigl\{ \{x\colon x\badjac y\}\colon y \badjac p, q \bigr\}
  \end{equation}
  is the line of\/ $\goth B$ through $p,q$. 
  Consequently, the class of lines of\/ $\goth B$ is definable in terms of the 
  binary collinearity $\badjac$ of\/ $\goth B$.
\end{thm}
\begin{proof}
  In view of \ref{prop:transtiv:GEN} without loss of generality
  we can assume that $p = [\btheta,\theta]$ and then \ref{lem:izoline=przekroj}
  yields the claim directly.
\end{proof}
%


\section{Symplectic affine polar spaces}
\subsection{Automorphisms}

Now, we return to the notation of Subsection \ref{subsec:def:apsy}.

In view of \ref{prop:afpol2af}, 
  $\Aut(\fixafpol) \subset \Aut(\fixaf)$.
Moreover, in view of \ref{exm:syaps}, 
as a particular instance of \ref{prop:aut:GEN} we 
get the following characterization.
\begin{prop}\label{prop:aut:af}
  Let $f\in\AF(\fixaf)$ be a linear (affine) automorphism of\/ \fixaf.
  The following conditions are equivalent.
  \begin{sentences}\itemsep-2pt
  \item
    $f$ is an automorphism of\/ \fixafpol.
  \item
    The map $f$ is given by the formula
    \begin{equation}\label{wz:aut:af}
      f([a,u]) = [\alpha a + v\circ u + b, \varphi(u) + w],\quad
      [a,u] \text{ is a vector of } {\field Y},
    \end{equation}
    where
    $\alpha\neq 0$ and $b$ are scalars, 
    $v, w$ are vectors of\/ $\field V$, 
    $\circ$ stands for the ``Cartesian" scalar product,
    and 
    $\varphi$ is a linear bijection of\/ $\field V$ such that
    \begin{enumerate}[a)]\itemsep-2pt
    \item\label{wrr1}
      $v \circ u = \eta(\varphi(u),w)$ for every vector $u$ of\/ $\field V$, and
    \item\label{wrr2}
      $\eta(\varphi(u_1),\varphi(u_2)) = \alpha \eta(u_1,u_2)$, for all vectors
      $u_1,u_2$ of\/ $\field V$.
    \end{enumerate}
  \end{sentences}
\end{prop}
\begin{note*}\small
  For a given map $\varphi$, the condition \rref{wrr2} uniquely determines
  the parameter $\alpha$.
  Similarly, for given $\varphi$ and $w$, the map 
  $u \longmapsto \eta(\varphi(u),w)$ is a functional in ${\field V}^\ast$ 
  and thus there is a vector $v$ that satisfies \rref{wrr1}.
\end{note*}
Note that the group $\Aut(\fixafpol)$ does not contain a transitive 
subgroup of translations:

\begin{lem}\label{lem:aut:tran}
  Let $[b,w]$ be a vector of\/ $\field Y$.
  Then $\tau_{[b,w]}\in\Aut(\fixafpol)$ iff $w=\theta$.
\end{lem}
\begin{proof}
  A translation as above is a linear map of the form \eqref{wz:aut:af}
  with the ``linear part" equal to the identity, in particular, with 
  $v = \theta$. From \rref{wrr1} we get $\eta(\varphi(u),w)\equiv 0$
  so, $w = \theta$.
\end{proof}

Similarly, from \ref{prop:aut:af} we get
\begin{lem}\label{lem:aut:lin}
  Let $\psi$ be a linear map of\/ ${\field Y}$. 
  Clearly, 
  $\psi\in\Aut(\fixaf)$.
  The following conditions are equivalent
  \begin{sentences}\itemsep-2pt
  \item
    $\psi\in\Aut(\fixafpol)$.
  \item\label{aut:lin:war2}
    There is a nonzero scalar $\alpha$ and a linear bijection $\varphi$ of\/ $\field V$
    such that
    \begin{enumerate}[\rm a)]\itemsep-2pt
    \item
      $\psi([a,u]) = [\alpha a,\varphi(u)]$ for every vector $[a,u]$ of\/ $\field Y$;
    \item
      $\eta(\varphi(u_1),\varphi(u_2)) = \alpha \eta(u_1,u_2)$
      for all vectors $u_1,u_2$ of $\field V$ i.e. $\varphi$ preserves $\eta$.
    \end{enumerate}
  \end{sentences}
\end{lem}

Nevertheless, as a direct consequence of  \ref{prop:transtiv:GEN} 
we infer that the structure \fixafpol\ is homogeneous:
\begin{prop}
  The group $\Aut(\fixafpol)\cap\GL({\field Y})$ acts transitively on 
  the point set of \fixafpol.
\end{prop}

Let $f$ have the form \eqref{wz:aut:af} such that \rref{wrr1} and \rref{wrr2} hold.
Let us write $\alpha = |\varphi|$; then \rref{wrr2} assumes form 
$\eta(\varphi(u_1),\varphi(u_2)) = |\varphi|\eta(u_1,u_2)$.
Clearly, the map $\varphi\longmapsto |\varphi|$ is a homomorphism of the group 
of ``admissible" $\varphi$'s and the multiplicative group of $\goth F$.
With this notation, the formula \eqref{wz:aut:af} can be rewritten in the form
\begin{equation}
  f([a,u]) = \bigl[ |\varphi|a + \eta(\varphi(u),w) + b, \varphi(u) + w \bigr].
\end{equation}
Consequently, every $f\in\Aut(\fixafpol)\cap\GL({\field Y})$
can be identified with a triple 
$(b,w,\varphi)\in F\times {\field V} \times \Aut(\eta)$, where 
$\Aut(\eta)$ is the group of automorphisms of $\perp_\eta$ on $\field V$,
i.e. 
\begin{ctext}
  $\Aut(\eta) = \left\{ \varphi\in\GL({\field V}) \colon 
  \exists {\alpha \in F} \; \left[\alpha\neq 0 \Land 
  \forall {u_1,u_2} \eta(\varphi(u_1),\varphi(u_2)) = \alpha\eta(u_1,u_2)\right] \right\}$.
\end{ctext}

Let $f_i$ be associated with the triple $(b_i,w_i,\varphi_i)$ for $i=1,2,3$
and let $f_3 = f_2 f_1$.
First, we directly get $\varphi_3 = \varphi_2\varphi_1$.
From well known formula of (analytical) affine geometry
$(\tau_{\omega_2}\psi_2)(\tau_{\omega_1}\psi_1) = 
\tau_{\omega_2 + \psi_2(\omega_1)}(\psi_2\psi_1)$, where $\psi_1,\psi_2$ are 
linear bijections we obtain
$w_3 = \varphi_2(w_1) + w_2$ and 
$b_3 = |\varphi_2|b_1 + \eta(\varphi(w_1),w_2) + b_2$.
The obtained rules for transformation of parameters $\varphi,w$ are analogous
to the transformation rules of the direct product $\Tr(\field V)\rtimes \Aut(\eta)$,
but the transformation rule of $b$ is more complex.

\begin{lem}\label{lem:aut:cialo}
  Let $\sigma\in\Aut(\goth F)$, let us fix a natural basis of\/ $\field V$
  and for $u = [\alpha_1,\dots,\alpha_n]$ in $\field V$
  let us set $\sigma(u) = [\sigma(\alpha_1),\dots,\sigma(\alpha_n)]$.
  Finally, we set $\sigma^\ast([a,u]) = [\sigma(a),\sigma(u)]$
  for $[a,u]$ in $\field Y$.
  Then $\sigma^\ast\in\Aut(\fixafpol)$.
\end{lem}
\begin{proof}
  It suffices to note that for any vectors $u_1,u_2$ in $\field V$
  we have $\eta(\sigma(u_1),\sigma(u_2)) = \sigma(\eta(u_1,u_2))$
  and use \ref{lem:anal:colin}.
\end{proof}

Summing up 
\ref{prop:aut:af}
and \ref{lem:aut:cialo}
we obtain a characterization of the group $\Aut(\fixafpol)$.
\begin{cor}
  The group $\Aut(\fixafpol)$ consists of all the maps
  \begin{ctext}
    ${\field Y}\ni[a,u] \longmapsto \bigl[ \alpha \sigma(a) + \eta(\varphi(u), w) + b, \varphi(a) + w \bigr]$
  \end{ctext}
  with a nonzero scalar $\alpha$, a scalar $b$ in\/ $\goth F$,
  an automorphism $\sigma$ of\/ $\goth F$,
  and a $\sigma$-semilinear bijection $\varphi$ of\/ $\field V$
  such that\/ $\eta(\varphi(u_1),\varphi(u_2)) = \alpha\sigma(\eta(u_1,u_2))$
  for $u_1, u_2\in {\field V}$.
\end{cor}

The natural question appears what is the ``metric" geometry of our symplectic
affine spaces i.e. what
are characteristic relations defined on the point universe of \fixaf\ that
characterize our geometry (except $\adjac$, of course, which {\em is sufficient},
but is more {\em affine} than {\em metric} in spirit).
It is clear that no relation that is invariant under all the translations
can be used here. In particular, no line orthogonality can be used.

For pairs $(p_1,p_2)$, $(p_3,p_4)$ of points of \fixaf\ we define
\begin{equation}\label{def:przystawanie}
  p_1p_2 \equiv p_3p_4 :\iff \rho(p_1,p_2) = \rho(p_3,p_4).
\end{equation}
Clearly, the relation $\equiv$ is an equivalence relation.
The following relation is crucial. For distinct $p_1,p_2$ and
arbitrary $p$ we have
\begin{ctext}
  $p_1p_2 \equiv pp$ iff $p_1 \adjac p_2$.
\end{ctext}
Since $\adjac$ is expressible in terms of $\equiv$, each automorphism of $\equiv$
preserves $\adjac$, so, in view of \ref{prop:afpol2af}, it is an automorphism
of \fixafpol.
From \ref{prop:aut:af}, \ref{prop:afpol2af}, and the above we can directly compute
\begin{prop}
  The following conditions are equivalent.
  \begin{sentences}\itemsep-2pt
  \item
    $f\in\Aut(\struct{Y,\equiv})$. 
  \item
    $f\in\Aut(\fixafpol)$. 
  \item
    There is a nonzero scalar $\alpha$ and 
	an automorphism $\sigma$ of\/ $\goth F$ such that we have
    $\rho(f(p_1),f(p_2)) = \alpha\sigma\rho(p_1,p_2)$ for all points $p_1,p_2$ of\/ \fixaf.
  \end{sentences}
\end{prop}


\subsection{Bisectors and symmetries}

In this section we also follow notation of Subsection \ref{subsec:def:apsy}.
Next, is a technical fact which we need later.

\begin{lem}\label{lem:pomo1}
  For fixed $w,\alpha,\beta$ the set 
  \begin{equation}\label{eq:hipcia}
    {\cal Z} = \bigl\{ [a,u]\colon \eta(w,u) = \beta + \alpha a \bigr\}
  \end{equation}	
  either is empty ($w = \theta$, $\alpha = 0$, $\beta\neq 0$),
  or it is the point set of $\fixaf$ ($w = \theta$, $\alpha,\beta = 0$),
  or it is a hyperplane of\/ \fixaf.
\end{lem}
\begin{proof}
  From \ref{lem:pomo1X},  the set $\cal Z$ is a subspace of \fixaf.
  From \ref{lem:pomo1Y}, if $w\neq\theta$ then
  $\dim({\cal Z}) = \dim(\fixaf)-1$, which is our claim.
\end{proof}

The relation $\equiv$ defined by \eqref{def:przystawanie}
has some properties of an abstract ``equidistance relation"
or a ``segment congruence",
but note that it is not associated with any norm.
From definition we have $\rho(p_2,p_1) = - \rho(p_1,p_2)$, which gives
\begin{ctext}
  $p_1p_2 \equiv p_3p_4$ iff $p_2p_1 \equiv p_4p_3$, but
  $p_1p_2 \equiv p_2p_1$ iff $p_1 \adjac p_2$.
\end{ctext}

Another similarity concerns bisector hyperplanes.
In the context of a ``metric'' geometry 
determined by an equidistance relation $\equiv$, 
with a pair of points $p_1,p_2$ one can associate ``bisectors'' of $p_1,p_2$:
\begin{equation}\label{eq:bisectors}
  \bisect(p_1,p_2) = 
  \left\{ p\colon p_1 p \equiv p_2 p \right\}
  \quad\text{ and }\quad
  \bisecm(p_1,p_2) = 
  \left\{ p\colon p_1 p \equiv p p_2 \right\}.
\end{equation}
(If in a geometry we have $pq\equiv qp$ for all points $p,q$, 
then there is no need to distinguish between these two types of bisectors.
In our case the distinction is necessary.)
The following is evident.
\begin{rem}\label{rem:bistryw}
  Let $p_0$ be a point of \fixaf.
  Then $\bisect(p_0,p_0)$ is the point set of \fixaf\
  and $\bisecm(p_0,p_0)$ is the hyperplane of \fixaf\ consisting of the 
  points collinear in \fixafpol\ with $p_0$.
\end{rem}

Let 
$p_i = [a_i,u_i]$ for $i =1,2$, directly from definitions
\eqref{def:metryka} and \eqref{def:przystawanie} we compute the following
\begin{ctext}
  $\bisect(p_1,p_2) =
  \left\{ [a,u] \colon \eta(u_1-u_2,u) = a_1 - a_2  \right\}$
and
  $\bisecm(p_1,p_2) =
  \left\{ [a,u]\colon \eta(u_1+u_2,u) = (a_1 + a_2) - 2 a \right\}$.
\end{ctext}
Thus by \ref{lem:pomo1}, we can complete \ref{rem:bistryw} as follows.
\begin{prop}
  Let $p_1,p_2$ be a pair of distinct points of \fixafpol.
  $\bisect(p_1,p_2)$ is empty iff $\srodek$ is the direction 
  of $\LineOn(p_1,p_2)=:L$. When $\srodek$ is not the direction of $L$ 
  then $\bisect(p_1,p_2)$, 
  and in any case $\bisecm(p_1,p_2)$ both are hyperplanes of \fixaf.
\end{prop}
What is more intriguing is that an analogue of a ``sphere"
\begin{ctext}
  $\left\{ p \colon p_1 p \equiv p_1 p_2 \right\} =
  \left\{ [a,u]\colon \eta(u_1,u) = (\eta(u_1,u_2) - a_2) + a \right\}$
\end{ctext}
is a hyperplane of\/ \fixaf\ as well.

Comparing the properties of $\equiv$ defined here and of the equidistance
of a metric affine space (characterized in \cite{szroeder}) we see that
the only one axiom of \cite{szroeder} that is not 
valid here is this which states that opposite sides of a parallelogram 
are congruent (i.e. stating that a translation is an ``isometry'').
Some other (we believe: interesting) properties are valid, instead.

Let us write $\oplus$ for the (affine) midpoint operation defined in \fixaf.
Simple computation gives
\begin{lem}
  $p_1 p_1\oplus p_2 \equiv p_1\oplus p_2 p_2$ for each pair $p_1,p_2$ of 
  points of \fixafpol. Consequently, $p_1 \oplus p_2 \in \bisecm(p_1,p_2)$.
\end{lem}

Given a hyperplane $H$ we say that $p_1,p_2$ are {\em symmetric} wrt. $H$
when $H$ is the bisector hyperplane of $p_1,p_2$. 
In a metric affine geometry this notion can be used to define the axial symmetry.
In our geometry we have, formally, two notions of a pair symmetric under a hyperplane.

\begin{lem}\label{lem:bisecty:par}
  Let $q = [b,w]$ be an arbitrary vector of $\field Y$.
  Then 
  \begin{ctext}
    $\bisect(p_1,p_1 + q) = \bisect(p_2,p_2 + q)$ 
    and
    $\bisecm(p_1,2q - p_1) = \bisecm(p_2,2q - p_2)$
  \end{ctext}
  for any pair $p_1,p_2$ of points of \fixafpol.
\end{lem}
\begin{proof}
  Let $p_i = [a_i,u_i]$ and $r = [c,y]$. 
  Clearly, $\bisect(p_1,p_1 + q)= \bisect(p_2,p_2+q)$
  immediately follows from the condition
  \begin{ctext}
    $\rho(p_1+q,r)-\rho(p_1,r) = \rho(p_2+q,r)-\rho(p_2,r)$
    for each $r$.
  \end{ctext}
  Let us verify this condition. We compute:
  \begin{math}
  \rho(p_i + q,r) = \eta(u_i+w,y) - (a_i + b -c),\;
  \rho(p_i,r) = \eta(u_i,y) - (a_i - c),
  \rho(p_i+q,r) - \rho(p_i,r) = \eta(w,y) - b.
  \end{math}
  This closes the proof of our first equality.
  The second one is computed analogously.
\end{proof}

Let us prove an auxiliary fact
\begin{fact}\label{fct:pomo:2}
  Let $w,u$, $\alpha,\beta$ be fixed.
  Suppose that $\eta(w,u) = \alpha \iff \eta(y,u) = \beta$
  holds for all vectors $u$. Then $y = \gamma w$ and $\beta = \gamma\alpha$
  for a nonzero scalar $\gamma$.
\end{fact}
\begin{proof}
  Take $w',y'$ such that $\eta(w,w') = \alpha$ and $\eta(y,y') = \beta$.
  From the assumptions we get
    $\eta(w,u-w') = 0 \iff \eta(y,u-y') = 0$ 
  i.e. 
    $u-w' \in w^\perp \iff u-y' \in y^\perp$.
  This gives $u'+ w^\perp = y'+y^\perp$ from which $w^\perp = y^\perp$
  follows. Thus $y\parallel w$, so $y = \gamma w$ for some $\gamma$.
  Finally, $\beta = \gamma\alpha$ is directly computed.
\end{proof}

\begin{lem}\label{lem:rownebisy}
  Let $p_i= [a_i,u_i]$ and $p'_i = [a'_i,u'_i]$ for $i=1,2$. Then
  \begin{sentences}\itemsep-2pt
  \item\label{rownebisy:t}
    $\bisect(p_1,p_2)= \bisect(p'_1,p'_2)$ iff $(p_2 - p_1)\gamma = p'_2 - p'_1$
    for a nonzero scalar $\gamma$, and
  \item\label{rownebisy:m}
    $\bisecm(p_1,p_2)= \bisecm(p'_1,p'_2)$ iff  
    $p_2 + p_1 = p'_2 + p'_1$.
  \end{sentences}
\end{lem}
\begin{proof}
  Similarly as in \ref{lem:bisecty:par} we compute that
  $\bisect(p_1,p_2) = \bisect(p'_1,p'_2)$ is equivalent to the condition
  $(\forall u)\big[\eta(u_1-u_2,u) = (a_1-a_2) \iff \eta(u'_1-u'_2,u) = (a'_1-a'_2)\big]$.
  Applying \ref{fct:pomo:2} we get \eqref{rownebisy:t}.
  \par
  Analogously, 
  $\bisecm(p_1,p_2) = \bisecm(p'_1,p'_2)$ is equivalent to the condition
  \begin{ctext}
    $(\forall u)(\forall a)
    \big[\eta(u_1+u_2,u) = (a_1+a_2 - 2a) \iff \eta(u'_1+u'_2,u) = (a'_1+a'_2-2a)\big]$.
  \end{ctext}
  Substituting $a=0$ to the above we get $u'_1+u'_2 = (u_1+u_2)\gamma$ for some 
    $\gamma\neq 0$ and 
    $a'_1 + a'_2 = (a_1 + a_2)\gamma$.
  Considering again the above with $a= 1$ we obtain $\gamma = 1$.
  This, together with \ref{lem:bisecty:par} proves \eqref{rownebisy:m}.
\end{proof}
Note that from \ref{lem:rownebisy} in case 
\eqref{rownebisy:t}, right-to-left 
with $\gamma=1$ we get \ref{lem:bisecty:par}.

The statements \ref{lem:bisecty:par} and \ref{lem:rownebisy} can be summarized in the
following
\begin{thm}\label{thm:symmetry}
  Let $H$ be a hyperplane of \fixaf.
  The relation $p_1 \mathrel{\sigma^{\sf t}_H} p_2$ 
  defined by the condition $\bisect(p_1,p_2) = H$ 
  is not a function; nevertheless,
  whenever $p_1\mathrel{\sigma^{\sf t}_H} p_2$ holds, the translation
  which maps $p_1$ onto $p_2$ is a subset of $\sigma^{\sf t}_H$.
  \par
  The relation $p_1 \mathrel{\sigma^{\sf m}_H} p_2$ 
  defined by the condition $\bisecm(p_1,p_2) = H$ 
  is either void or it is
  the central symmetry of \fixaf\ with the centre $p_1\oplus p_2$.
\end{thm}

Finally, a result of a simple computation which has interesting consequences:
\begin{prop}\label{prop:bisec=polar}
  Let $p_i = [a_i,u_i]$ be distinct points of \fixaf,  
  $q = p_1\oplus p_2$, and $\vartheta$ be the direction of the line $\LineOn(p_1,p_2)$
  (i.e. $q = [\frac{a_1+a_2}{2},\frac{u_1+u_2}{2}]\in {\field Y}$,  
  $\vartheta = [0,a_2-a_1,u_2-u_1]\in {\field W}$).
  Then
  $$
    \textstyle{\bisecm(p_1,p_2)} = 
    \{ r\colon r\text{ is a point of }\fixaf, q \adjac r  \},\;
    \textstyle{\bisect(p_1,p_2)} = 
    \{ r\colon r\text{ is a point of }\fixaf, \vartheta\perp_\xi r \}.
  $$
  Consequently,
  $\bisecm(p_1,p_2)$ is the restriction 
  of the hyperplane $(p_1\oplus p_2)^{\perp}$ of \fixproj\ polar of $p_1\oplus p_2$
  to the point set of \fixaf.
  Analogously, $\bisect(p_1,p_2)$ is the restriction of the hyperplane
  which has the direction $[{\LineOn(p_1,p_2)}]_\parallel$ 
  as its pole under the underlying null-polarity.
\end{prop}


\bigskip
\begin{small}
\noindent
Authors' address:
\\
Krzysztof Pra{\.z}mowski,
Mariusz {\.Z}ynel
\\
Institute of Mathematics, University of Bia{\l}ystok
\\
ul. Akademicka 2, 15-267 Bia{\l}ystok, Poland
\\
\verb+krzypraz@math.uwb.edu.pl+,
\verb+mariusz@math.uwb.edu.pl+
\end{small}

\end{document}